\documentclass[a4paper,11pt,reqno]{amsart}
\usepackage{amsmath,amssymb,amsthm,enumerate,
cite
}

\nonstopmode
\numberwithin{equation}{section}
\newtheorem{lemma}{Lemma}[section]
\newtheorem{theorem}{Theorem}[section]

\newtheorem{corollary}{Corollary}[section]
\newtheorem{remark}{Remark}[section]
\newtheorem*{note}{Note}
\newtheorem{conjecture}{Conjecture}
\newtheorem*{problem}{Problem}
\newtheorem{proposition}{Proposition}
\numberwithin{equation}{section}
\DeclareMathOperator{\Log}{Log}

\title{Pick Functions Related to the Multiple Gamma Functions of order $n$}
\author[]{Sourav Das}
\address{Department of Mathematics, Indian Institute of Technology, \newline Roorkee-247 667, Uttarakhand, India }
\email{das90dma@iitr.ac.in}
\author[]{A. Swaminathan}
\address{Department of Mathematics, Indian Institute of Technology, \newline Roorkee-247 667, Uttarakhand, India }
\email{swamifma@iitr.ac.in}
\subjclass[2010]{30E20, 33B15}
\begin{document}
\begin{abstract}
Let $G_n$ be the Barnes multiple Gamma function of order $n$ and the function $f_n(z)$ be defined as
\begin{align*}
f_n(z)=\dfrac{\log G_n(z+1)}{z^n\Log z},\quad z\in \mathbb{C}\setminus (-\infty,0].
\end{align*}
In this work, a conjecture to find the Stieltjes representation
is proposed such that $f_n(z)$ is a Pick function.
The conjecture is established for the particular case $n=3$ by examining
the properties of $f_3(z)$.
\end{abstract}
\keywords{Multiple Gamma function, Pick function, Stieltjes integral representation, logarithm function}
\subjclass[2010]{33B15, 41A60}

\maketitle

\pagestyle{myheadings} \markboth{Sourav Das and A. Swaminathan}
{Pick functions related to Multiple Gamma functions of order $n$}

\section{Introduction}
The class of functions $g(z)=u(z)+iv(z)$ analytic in the upper-half plane with positive imaginary part are
called Pick class \cite[p. 18]{Donoghue74_monograph} and denoted by $\mathbb{P}$.
A Pick function $f(z)\in\mathbb{P}$  can be extended by reflection to holomorphic functions in
$\mathbb{C}\setminus \mathbb{R}$ and they have the integral representation as following
\cite[p. 508]{BergPedersen02_PickGamma}
\begin{align}\label{eqn:Pick-DefnIntgralRep}
f(z)=az+b+\int_{-\infty}^{\infty}\left(\dfrac{1}{t-z} -\dfrac{t}{t^2+1} \right)d\mu(t),
\end{align}
where $a\geq0,\;b\in\mathbb{R}$ and $\mu$ is a nonnegative Borel measure in $\mathbb{R}$ satisfying
\begin{align*}
\int_{-\infty}^{\infty}\dfrac{d\mu(t)}{t^2+1}<\infty,
\end{align*}
with
\begin{align}\label{eqn:PickValue-a-b}
a=&\lim_{y\rightarrow\infty}\dfrac{f(iy)}{(iy)},\quad b=\Re f(i),\\ \label{eqn:PickValue-mu}
\mu=&\lim_{y\rightarrow 0+}\dfrac{\Im f(t+iy)dt}{\pi}.
\end{align}
It is well known that the limit for $\mu$ is in vague topology  \cite[p. 32]{Donoghue74_monograph}
(see also \cite[p. 508]{BergPedersen02_PickGamma}, \cite[p. 96]{Pedersen07_ETNA_PickNegativeZero}).
The reciprocal of a Stieltjes function belongs to $\mathbb{P}$ is observed in \cite[p. 507]{BergPedersen02_PickGamma}.
For more information on Pick functions one can see
\cite{Donoghue74_monograph,BergPedersen01_MonotoneGamma,BergPedersen02_PickGamma,Berg2001_
PickEntropy,BergPedersen11_PickGammaUnitball,Pedersen03_JCAM_doubleGammaPick,Pedersen05_JCAM_CanonicalPick}.

One of the recent interest, among researchers, is to find the members of the class of Pick functions.
For example, the second author posed the problems of finding Pick functions related to generalized polylogarithm
and generalized hypergeometric functions in \cite[p.91]{Berg15_ITSF_OpenProblem}.
Functions that are related to the logarithm of Gamma functions are of special interest.

In 1899, E.W. Barnes \cite{Barnes1899_Gfunction} introduced multiple Gamma functions,
denoted by $\Gamma_n(z)$ or $G_n(z)$ ,
using multiple Hurwitz zeta functions.
The Gamma function of order $n$, denoted by $\Gamma_n$ can be defined by the recurrence relation
\begin{align*}
\Gamma_{n+1}(z+1)=\dfrac{\Gamma_{n+1}(z)}{\Gamma_n(z)}; \quad \Gamma_1(z)=\Gamma(z); \quad \Gamma_n(1)=1.
\end{align*}
Here $\Gamma(z)$ denotes the usual Gamma function, studied by Euler, which is defined by
\begin{align*}
\Gamma(z):=\int_0^{\infty}e^{-t}t^{z-1} dt, \qquad \qquad {\rm Re \,} z>0.
\end{align*}
We remark that equivalent forms of $\Gamma(z)$ exist in the literature, (see \cite{Choi14_MultipleGammaSurvey} for details).

By defining
\begin{align*}
\Gamma_n(z)=( G_n(z) )^{(-1)^{n-1}},\quad n\in \mathbb{N}.
\end{align*}
it is easy to see that
$$
G_1(z)=\Gamma_1(z)=\Gamma(z),\quad G_2(z)=\dfrac{1}{\Gamma_2(z)},\quad G_3(z)=\Gamma_3(z).
$$
These $G_n(z)$'s are known as multiple Gamma functions of order $n$ and studied first in detail
by Vign\'eras \cite{Vigneras78_CRNS_Gamma} in 1978 (see also \cite{Barnes1899_Gfunction,Barnes1904_multipleGamma}).
Note that, in \cite{Barnes1904_multipleGamma} Barnes
had shown that $G_2(z)=\dfrac{1}{\Gamma_2(z)}=G(z)$, where $G(z)$ is the Barnes $G$- function.

Clearly, these $G_n(z)$'s satisfy the recurrence relation
\begin{align}\label{eqn:G-n-RecReln}
G_{n+1}(z+1)=G_{n+1}(z)G_{n}(z), \quad G_n(1)=1, \quad z\in \mathbb{C},\; n\in \mathbb{N}.
\end{align}

Further information about $G_n(z)$, required for the results given in this paper, are provided in Section $\ref{Sec:proof of G-n}$.
For various applications of $G_n(z)$  we refer to \cite{Choi14_MultipleGammaSurvey} and references therein.

Consider the function
\begin{align*}
f_n(z)=\dfrac{\log G_n(z+1)}{z^n\Log z},\quad z\in \mathbb{C}\setminus (-\infty,0],\quad n\in\mathbb{N}.
\end{align*}

Berg and Pedersen already proved that $f_1(z)\in \mathbb{P}$  and found its Stieltjes representation in
\cite{BergPedersen01_MonotoneGamma,BergPedersen02_PickGamma}.
In \cite[p. 364]{Pedersen03_JCAM_doubleGammaPick} Pedersen found the Stieltjes representation of $f_2(z)$ and proved that it
belongs to $\mathbb{P}$.

Motivated by the earlier results on $f_1(z)$ and $f_2(z)$, we are interested in finding the Stieltjes representation for
$f_n(z)$ so that it belongs to the class of Pick functions. The problem we are interested in is given below as a conjecture as
it is challenging to establish this as a result.

\begin{conjecture}\label{conj:f-n-Pick}
The function
\begin{align*}
f_n(z)=\dfrac{\log G_n(z+1)}{z^n\Log z},\quad z\in \mathbb{C}\setminus (-\infty,0], \qquad n\in {\mathbb{N}},
\end{align*}
is a Pick function with Stieltjes representation
$$\dfrac{\log G_n(z+1)}{z^n\Log z}=\dfrac{1}{n!}-\int_0^{\infty}\dfrac{\tilde{d}_n(-t)}{t+z}dt.$$
where
\begin{align} \nonumber
\tilde{d}_n(t)&=\left\{
         \begin{array}{ll}
           -\dfrac{(-1)^{n-1}N_n(|t|)\log|t|+\log|G_n(t+1)|}{t^n((\log|t|)^2+\pi^2)}, & \hbox{$t<0$;} \\
           0, & \hbox{$t\geq0$.}
         \end{array}
       \right.
,\\ \label{eqn:counting-function-Pick-Gamma}
N_n(t)&= \dfrac{[t]([t]+1)([t]+2)\cdots([t]+n-1)}{n!}  \qquad \qquad \mbox{ for } t>0.
\end{align}

\end{conjecture}

As mentioned earlier, this result is established for $n=1$ in \cite{BergPedersen02_PickGamma}(see also \cite{BergPedersen01_MonotoneGamma}
and for $n=2$ in \cite{Pedersen03_JCAM_doubleGammaPick}. In Section $\ref{Sec: triple-Gamma-Pick-results}$, we establish the result for $n=3$.
Even though it is difficult to establish the above conjecture for general $n$, numerical and
graphical evidences suggest that the conjecture is true.
We provide below some results to support our claim. The first one we have in this line is due to Theorem 3 of
\cite[p. 436]{UenoNishizawa97_qMultiGamma} which we state as a lemma.
\begin{lemma}
\begin{align*}
\lim_{n\rightarrow\infty}\dfrac{G_n(x+1)}{x^n\log x}=\dfrac{1}{n!}
\end{align*}
\end{lemma}

The value of $\log G_n(z)$ at a regular point, in the upper half plane, is provided by the following result.
\begin{proposition}\label{Prop::log-Gamma-n-z-t}
 We have for any $k,n\in\mathbb{N}$,
$$
\lim_{\begin{array}{c} z\rightarrow t  \\\Im z>0\end{array}} \log G_n(z)=\log|G_n(t)|+i(-1)^n\pi \dfrac{(k)_n}{n!}
$$
for $t\in (-k,-k+1)$, where $(k)_n=k(k+1)(k+2)\cdots(k+n-1)$ is the Pochhammer symbol.
\end{proposition}

The above result, for the case $n=1$ is given in \cite{BergPedersen01_MonotoneGamma}.
We state this case explicitly as it is used in further discussion.

\begin{lemma}{\rm{ \cite[Lemma 2.1]{BergPedersen01_MonotoneGamma}}}\label{lemma:log-Gamma-z-t}
We have for any $k\geq 1$,
$$
\lim_{\begin{array}{c} z\rightarrow t  \\\Im z>0\end{array}} \log \Gamma(z)=\log|\Gamma(t)|-i\pi k
$$
for $t\in (-k,-k+1)$  and $$ \lim_{\begin{array}{c} z\rightarrow t  \\\Im z>0\end{array}}|\log\Gamma(z)|=\infty$$
for $t=0,-1,-2,\ldots$.
\end{lemma}

The following result would be useful in establishing the Stieltjes representation for $f_n(z)$.

\begin{proposition}\label{Prop:Im-f-n-in-k-k+1}
For any $n,k\in\mathbb{N}$ we have
$$
\lim_{\begin{array}{c} z\rightarrow t  \\\Im z>0\end{array}}\Im f_n(z)
    =\pi\dfrac{\frac{(-1)^n}{n!}(k-1)_n\log|t|-\log|G_n(t+1)|}{t^n((\log|t|)^2+\pi^2)}
$$
for $t\in(-k,-k+1)$.
\end{proposition}

The paper is organized as follows.
Proofs of Proposition $\ref{Prop::log-Gamma-n-z-t}$ and Proposition $\ref{Prop:Im-f-n-in-k-k+1}$  and
further details on $G_n(z)$ are given in Section $\ref{Sec:proof of G-n}$.
In Section $\ref{Sec: triple-Gamma-Pick-results}$ results for the case $n=3$ related to Conjecture $\ref{conj:f-n-Pick}$ are provided.
Concluding remarks with related problems for further research are provided in Section $\ref{sec:Remarks-G-n-Pick}$.

\section{Properties of Multiple Gamma function}\label{Sec:proof of G-n}

A result on obtaining the unique meromorphic function $G_n(z)$, by the analogy of Bohr-Mollerup theorem,
given in \cite[Theorem 1.2, p.100]{Choi14_MultipleGammaSurvey}), is as follows.
\begin{lemma}{\rm \cite{Choi14_MultipleGammaSurvey}}\label{lemma:G-n-Bohr-Mollerup}
For all $n\in {\mathbb{N}}$, there exists a unique meromorphic function $G_n(z)$ satisfying each of the following properties
\begin{enumerate}[1.]
\item $G_n(z+1) = G_{n-1}(z)G_n(z)$, \qquad $z\in {\mathbb{C}}$.
\item $G_n(1)=1$ and $G_0(x)=x$.
\item For $x\geq 1$, $G_n(x)$ are infinitely differentiable and
\begin{align*}
\dfrac{d^{n+1}}{dx^{n+1}}\left(\frac{}{}\log G_n(x)\right) \geq 0.
\end{align*}
\end{enumerate}
\end{lemma}

By analyzing Lemma $\ref{lemma:G-n-Bohr-Mollerup}$, it is easy to see that $\displaystyle (\Gamma_n(z))^{-1}$ is an entire
function with zeros at $z=-k$, $k\in {\mathbb{N}}\cup\{0\}$ with multiplicities given by
\begin{align}\label{eqn:G-n-multiplicity-zeros}
\left(
\begin{array}{ccc}
n+k-1 \\
n-1
\end{array}
\right)
\qquad n\in{\mathbb{N}}, \, k\in {\mathbb{N}}\cup\{0\}.
\end{align}

Using \eqref{eqn:G-n-multiplicity-zeros}, the following explicit form of the multiple Gamma function of order $n$,
in terms of Weierstrass canonical product of $\Gamma_n$, is given in \cite{Choi14_MultipleGammaSurvey}.
{\small{
\begin{align}\label{eqn:G-n-canonical-product}
\Gamma_n(1+z) = \exp[Q_n(z)] \prod_{k=1}^{\infty} \left(\left(1+\dfrac{z}{k}\right)^{\left(
\begin{array}{ccc}
n+k-2 \\
n-1
\end{array}
\right)
}
\exp\left[
\left(
\begin{array}{ccc}
n+k-2 \\
n-1
\end{array}
\right)
\left(
\sum_{j=1}^{n}
\dfrac{(-1)^{j-1}}{j}\dfrac{z^j}{k^j}
\right)
\right]
\right),
\end{align}
}}
where $Q_n(z)$ is a polynomial of degree $n$ given by
\begin{align*}
Q_n(z):= & (-1)^{n-1}\left[-zA_n(1)+\sum_{k=1}^{n-1}\dfrac{p_k(z)}{k!}\left(\frac{}{}f_{n-1}^{(k)}(1)\right)\right],\\
f_n(z):= & -zA_n(1)+\sum_{k=1}^{n-1}\dfrac{p_k(z)}{k!}\left(\frac{}{}f_{n-1}^{(k)}(1)\right)+A_n(z),\\
A_n(z):= & \sum_{k=1}^{\infty}(-1)^{n-1}
\left(
\begin{array}{ccc}
n+k-2 \\
n-1
\end{array}
\right)
\left[
-\log\left(1+\dfrac{z}{k}\right)+\sum_{j=1}^{n}\dfrac{(-1)^{j-1}}{j}\dfrac{z^j}{k^j}
\right]
\end{align*}
and
\begin{align*}
p_n(z)=\dfrac{1}{n+1}\sum_{k=1}^{n+1}
\left(
\begin{array}{ccc}
n+1 \\
k
\end{array}
\right)
B_{n+1-k}z^k, \qquad n\in{\mathbb{N}}
\end{align*}
with $B_k$ being the Bernoulli numbers.

To prove Proposition $\ref{Prop::log-Gamma-n-z-t}$, we need a behaviour of $\log G_2$ which is given in the
following result.

\begin{lemma}\label{lemma:log-Gamma2-z-t}
 We have for any $k\geq 1$,
$$
\lim_{\begin{array}{c} z\rightarrow t  \\\Im z>0\end{array}} \log G_2(z)=\log|G_2(t)|+i\pi \dfrac{k(k+1)}{2}
$$
for $t\in (-k,-k+1)$  and $$ \lim_{\begin{array}{c} z\rightarrow t  \\\Im z>0\end{array}}|\log G_2(z)|=\infty$$
for $t=0,-1,-2,\ldots$ .

\end{lemma}
\begin{proof}
Using recurrence relation \eqref{eqn:G-n-RecReln} and Lemma \ref{lemma:log-Gamma-z-t} we have
{\small{
\begin{align*}
  \log G_2(z)&=\log G_2(z+k)-\sum_{l=0}^{k-1} \log G_1(z+l)\\
\Longrightarrow \lim_{\begin{array}{c} z\rightarrow t  \\\Im z>0\end{array}}\log G_2(z)&=\log G_2(t+k)-\sum_{l=0}^{k-1}(\log|G_1(t+l)|-i(k-l)\pi)\\
&=\log\left|\dfrac{G_2(t+k)}{G_1(t)G_1(t+1)\cdots G_1(t+k-1)} \right| +i\pi(k+(k-1)+\cdots+2+1) \\
&=\log|G_2(t)|+i\pi\dfrac{k(k+1)}{2}
\end{align*}
}}
The other part follows from the fact that $|\log G_2(z)|\geq \log|G_2(z)|$.
\end{proof}

\begin{proof}[Proof of Proposition $\ref{Prop::log-Gamma-n-z-t}$]
We have already shown that the statement is true for the case $n=1$ as given in
Lemma \ref{lemma:log-Gamma-z-t}.
The case $n=2$ is established in Lemma $\ref{lemma:log-Gamma2-z-t}$. Lemma $\ref{lemma:log-Gamma3-z-t}$
in Section $\ref{Sec: triple-Gamma-Pick-results}$ reflects the case $n=3$.

Let the statement be true for $n=m\in\mathbb{N}$. Then
\begin{align}\label{eqn:log-G-n-z-t}
\lim_{\begin{array}{c} z\rightarrow t  \\\Im z>0\end{array}} \log G_m(z)=\log|G_m(t)|+i(-1)^m\pi \dfrac{(k)_m}{m!}.
\end{align}
Now with the help of \eqref{eqn:log-G-n-z-t} we have
\begin{align*}
\log G_{m+1}(z)&=\log G_{m+1}(z)+\sum_{l=0}^{k-1}\log G_m(z+l)
\end{align*}
\noindent
$\displaystyle
\Longrightarrow
    \lim_{\begin{array}{c} z\rightarrow t  \\\Im z>0\end{array}} \log G_{m+1}(z)
$
\begin{align*}
&=\log G_{m+1}(t+k)-\sum_{l=0}^{k-1}\left(\log|G_{m}(t+l)|+i(-1)^m\pi\dfrac{(k-l)_m}{m!}  \right)\\
&=\log\left|\dfrac{G_{m+1}(t+k)}{G_m(t)G_m(t+1)\cdots G_m(t+k-1)}\right|+i(-1)^{m+1}\pi\sum_{l=0}^{k-1}\dfrac{(k-l)_m}{m!}
\end{align*}
Again with the help of southeast diagonal sum property (see \cite[eq. $(4.1.6a)$]{Gross08_Book} for details) and a simple computation we obtain
\begin{align*}
\sum_{l=0}^{k-1}\dfrac{(k-l)_m}{m!}
=\sum_{l=1}^k
\left(
\begin{array}{ccc}
l+m-1 \\
m
\end{array}
\right)
=
\left(
\begin{array}{ccc}
m+k \\
m+1
\end{array}
\right)
=\dfrac{(k)_{m+1}}{(m+1)!}.
\end{align*}
Therefore,
\begin{align*}
\lim_{\begin{array}{c} z\rightarrow t  \\\Im z>0\end{array}} \log G_{m+1}(z)=\log|G_{m+1}(t)|+i(-1)^{m+1}\pi\dfrac{(k)_{m+1}}{(m+1)!}
\end{align*}
which shows that the statement is true for $n=m+1$ when it is true for $n=m$.

Hence by the principle of mathematical induction the statement of the proposition is true for all $n\in\mathbb{N}$.
\end{proof}

\begin{note}
From Proposition $\ref{Prop::log-Gamma-n-z-t}$, we recall the remark
given in \cite[p.223]{BergPedersen01_MonotoneGamma} that the holomorphic branch of $\log H$  is real
on the positive half line, whenever the meromorphic function $H$ in ${\mathbb{C}}$ with zeros and
poles on $(-\infty,0]$ is real and positive on the positive half line. Further, the limit of
$\log H$ at a regular point is $i\pi$ multiplied by the number of zeros minus the number of poles in $(t,0]$
counted according to the multiplicity.
\end{note}

\begin{proof}[Proof of Proposition $\ref{Prop:Im-f-n-in-k-k+1}$]
Given that
$\displaystyle f_n(z)=\dfrac{\log G_n(z+1)}{z^n\Log z}$
Hence, for $t\in(-k,-k+1)$,\, $k\geq1$, we have
\begin{align*}
\lim_{\begin{array}{c} z\rightarrow t  \\\Im z>0\end{array}}f_n(z)
&=\dfrac{\log|G_n(t+1)|+i\pi\dfrac{(-1)^n}{n!}(k-1)_n}{t^n(\log|t|+i \pi)},\\
&=\pi\dfrac{\frac{(-1)^n}{n!}(k-1)_n\log|t|-\log|G_n(t+1)|}{t^n((\log|t|)^2+\pi^2)},
\end{align*}
which completes the proof.
\end{proof}

Note that for establishing Conjecture $\ref{conj:f-n-Pick}$ for $f_n(z)$, we need a result similar to Lemma $\ref{lemma:d3-t-value}$
and a result for growth at infinity. We are unable to establish the same at this stage. Hence in the next section, we provide
the results for the case $n=3$.

\section{The triple Gamma function related to Pick functions}\label{Sec: triple-Gamma-Pick-results}
In this section we provide the Stieltjes representation for the function
\begin{align*}
f_3(z)=\dfrac{\log \Gamma_3(z+1)}{z^3\Log z},\quad z\in \mathbb{C}\setminus (-\infty,0]
\end{align*}
so that $f_3(z)$ belongs to the class of Pick functions.

Note that, the triple Gamma function $\Gamma_3(z)$ can be defined by the Weierstrass canonical product form 
which will be a particular case of \eqref{eqn:G-n-canonical-product} for $n=3$.
Taking logarithm on both sides of that particular case, for $z\in\mathbb{C}\setminus(-\infty,0]$, we have
\begin{align*}
\log\Gamma_3(1+z)&=Dz^3+Ez^2+Fz+\log P_{\Gamma_3}(z)
\end{align*}
where
\begin{align}\label{eqn:P-Gamma3-Representation}
\log P_{\Gamma_3}(z)&=\sum_{k=1}^{\infty}\left(-\dfrac{k(k+1)}{2}\Log\left(1+\dfrac{z}{k}\right)
                +\dfrac{k+1}{2}z-\dfrac{1}{4}\left(1+\dfrac{1}{k}\right)z^2
                +\dfrac{1}{6k}\left(1+\dfrac{1}{k}  \right)z^3  \right),
\end{align}
with
{\small{
\begin{align} \label{eqn:P-Gamma3-representation-DEF-value}
D=-\dfrac{1}{6}\left(\gamma+\dfrac{\pi^2}{6}+\dfrac{3}{2}\right),
\qquad \qquad E=\dfrac{1}{4}\left(\gamma+\log(2\pi)+\dfrac{1}{2}\right) \qquad {\mbox{and}} \quad
F=\dfrac{3}{8}-\dfrac{\log(2\pi)}{4}-\log A
\end{align}
}}

With the help of maximum principle we will prove that $f_3$ has nonnegative imaginary part in the upper half plane.
For this we require to analyze the behaviour on the real line as well as the growth at infinity.
To prove our required results we need the following lemmas.

\begin{lemma}\label{lemma:P-Gamma3-bound}
There exist a constant and a sequence $\{r_n\}$ tending to infinity such that $$|\log P_{\Gamma_3}(z)|\leq \mbox{ Const } |z|^3\log|z|$$
holds for all $z\in\mathbb{C}\setminus(-\infty,0]$ of absolute value $r_n$.
\end{lemma}
\begin{proof}

Expression \eqref{eqn:P-Gamma3-Representation} can also be represented in the following canonical product form
\begin{align}\nonumber
P_{\Gamma_3}(z)
&=\prod_{k=1}^{\infty}\left(1+\dfrac{z}{k}\right)^{-\dfrac{k(k+1)}{2}}
    \exp{\left(\dfrac{k(k+1)}{2}\dfrac{z}{k}-\dfrac{k(k+1)}{2}\dfrac{z^2}{2k^2}+\dfrac{k(k+1)}{2}\dfrac{z^3}{3k^3}\right)}\\
\Rightarrow\dfrac{1}{P_{\Gamma_3}(z)}
&=\prod_{k=1}^{\infty}\left(1+\dfrac{z}{k}\right)^{\dfrac{k(k+1)}{2}}\exp\left\{{\dfrac{k(k+1)}{2} }
    \left( -\dfrac{z}{k}+\dfrac{z^2}{2k}-\dfrac{z^3}{3k} \right)\right\}
\end{align}
which implies $\dfrac{1}{P_{\Gamma_3}(z)}$ is of genus $3$.

Since $\dfrac{1}{P_{\Gamma_3}(z)}$ has zeros at $-k$ of multiplicity $\dfrac{k(k+1)}{2}$ for $k\geq1$
we obtain the zero counting function $n\equiv N_3(t)$ associated with the zeros of $\dfrac{1}{P_{\Gamma_3}(z)}$ as
 $$
N_3(t)=\dfrac{[t]([t]+1)([t]+2)}{6},\quad t>0,
$$
where $[t]$ denotes the integer part of $t$.

Hence using Appendix A and Proposition A.2 in \cite[p. 368]{Pedersen03_JCAM_doubleGammaPick} we have
\begin{align*}
\left|\log \dfrac{1}{P_{\Gamma_3}(z)}\right|&\leq \mbox{Constant} |z|^3 \log|z|\\
\Longrightarrow |\log P_{\Gamma_3}(z)|&\leq \mbox{Constant} |z|^3 \log|z|
\qedhere
\end{align*}
\end{proof}

\begin{lemma}\label{lemma:log-Gamma3-limit-x}
\begin{align*}
\lim_{x\rightarrow \infty}\dfrac{\log\Gamma_3(x+1)}{x^3\log x}=\dfrac{1}{6}.
\end{align*}
\end{lemma}
\begin{proof}
The Stirling formula for triple Gamma function \cite[p. 437]{UenoNishizawa97_qMultiGamma} is given by
\begin{align*}
\log\Gamma_3(x+1)
&\sim \left( \dfrac{x^3}{6}-\dfrac{x^2}{4}+\dfrac{1}{24} \right)\log(x+1)-\dfrac{11}{36}x^3+\dfrac{5}{24}x^2+\dfrac{x}{3}-\dfrac{13}{72}\\
&-\dfrac{x^2-x}{2}\zeta'(0)+\dfrac{2x-1}{2}\zeta'(-1)-\dfrac{1}{2}\zeta'(-2)\\
&+\dfrac{1}{12}\dfrac{1}{x+1}+\sum_{r=2}^{\infty}\{x^2-(6r-11)x+(4r^2-16r+16)\}
\end{align*}
Now dividing both sides by $x^3\log x$ and taking $x\rightarrow\infty$ we have the required result.
\end{proof}

\begin{lemma}\label{lemma:log-Gamma3-z-t}
 We have for any $k\geq 1$,
$$
\lim_{\begin{array}{c} z\rightarrow t  \\\Im z>0\end{array}} \log G_3(z)=\log|G_3(t)|-i\pi \dfrac{k(k+1)(k+2)}{6}
$$
for $t\in (-k,-k+1)$  and $$ \lim_{\begin{array}{c} z\rightarrow t  \\\Im z>0\end{array}}|\log G_3(z)|=\infty$$
for $t=0,-1,-2,\ldots$ .

\end{lemma}
\begin{proof}
Using recurrence relation \eqref{eqn:G-n-RecReln} and Lemma \ref{lemma:log-Gamma2-z-t} we have
\begin{align*}
  \log G_3(z)&=\log G_3(z+k)-\sum_{l=0}^{k-1} \log G_2(z+l)\\
\Longrightarrow
    \lim_{\begin{array}{c} z\rightarrow t  \\\Im z>0\end{array}}\log G_3(z)
&=\log G_3(t+k)-\sum_{l=0}^{k-1}\left(\log|G_2(t+l)|+i\pi \dfrac{(k-l)(k-l+1)}{2} \right)\\
&=\log\left|\dfrac{G_3(t+k)}{G_2(t)G_2(t+1)\cdots G_2(t+k-1)} \right| -\dfrac{i\pi}{2}\sum_{n=1}^k n(n+1) \\
&=\log|G_3(t)|-i\pi\dfrac{k(k+1)(k+2)}{6}
\end{align*}
The other part follows from the fact that $|\log G_3(z)|\geq \log|G_3(z)|$.
\end{proof}

\begin{lemma}\label{lemma:d3-t-value}
  Let $k\geq1$ and $\tilde{d}:\mathbb{R}\rightarrow\mathbb{R}\cup\{\infty \}$ be such that
\begin{align*}
  \tilde{d}_3(t)=\left\{
                 \begin{array}{lll}
                   \dfrac{(-1)^3\dfrac{k(k-1)(k+1)}{3!}\log|t|-\log|G_3(t+1)|}{t^3((\log|t|)^2+\pi^2)} , & \hbox{for $t\in(-k,-k+1)$;} \\
                   0, & \hbox{for $t\geq 0$;}\\
                   \infty, & \hbox{ for $t=-k$. }
                 \end{array}
               \right.
\end{align*}
  Then $\tilde{d}_3(t)\geq 0$.
\end{lemma}
\begin{proof}
It is enough to show that,  for $t\in(-k,-k+1)$, $\log|G_3(t+1)|+\dfrac{k(k-1)(k+1)}{6}\log|t|\geq0$, which
we obtain by induction hypothesis.

First we show that the statement is true for $k=1$. i.e;
$$
g(t)=\log|G_3(t+1)|\geq 0 \quad \mbox{ for } t\in(-1,0).
$$
Clearly, $g(t)$ is decreasing for $-1<t<0$, as
\newline
$g'(t)
$
{\small{
\begin{align*}
=-\dfrac{t^2}{2}\left( \gamma+\dfrac{\pi^2}{6}+\dfrac{3}{2}\right)+\dfrac{t}{2}\left(\gamma+\log2\pi+\dfrac{1}{2}\right)
+\left(\dfrac{3}{8}-\dfrac{\log2\pi}{4}-\log A \right)+\dfrac{t^3}{2}\sum_{k=0}^{\infty}\dfrac{k+2}{(k+1)^2(k+t+1)}
\end{align*}
}}
is negative for $-1<t<0$. Hence $ g(t)\geq g(0)=0$, as $G_3(1)=1$.

Now we will show that the statement is true for $k=2$. It is easy to see that
\begin{align*}
\log|G_3(t+1)|&=\left(D+\dfrac{1}{3}\right)t^3+\left(E-\dfrac{1}{2}\right)t^2+(F+1)t-\Log|1+t|\\
& \qquad \quad  +\sum_{k=2}^{\infty}\left(-\dfrac{k(k+1)}{2}\left(\Log\left|1+\dfrac{t}{k}\right|
    -\dfrac{t}{k}+\dfrac{1}{2}\dfrac{t^2}{k^2}-\dfrac{1}{3}\dfrac{t^3}{k^3}  \right)\right),
\end{align*}
where $D$, $E$ and $F$ are given, as in
\eqref{eqn:P-Gamma3-representation-DEF-value}.
Clearly
\begin{align*}
\sum_{k=2}^{\infty}\left(-\dfrac{k(k+1)}{2}\left(\Log\left|1+\dfrac{t}{k}\right|
-\dfrac{t}{k}+\dfrac{1}{2}\dfrac{t^2}{k^2}-\dfrac{1}{3}\dfrac{t^3}{k^3}  \right)\right)
\end{align*}
is positive, because
$$
\Log\left|1+\dfrac{t}{k}\right|<\dfrac{t}{k}-\dfrac{t^2}{2k^2}+\dfrac{t^3}{3k^3}\quad \mbox{ for }|t|<k.
$$
Now, considering the remaining part, it is enough to prove
$$
g(t):=\left(D+\dfrac{1}{3}\right)t^3+\left(E-\dfrac{1}{2}\right)t^2+(F+1)t-\Log|1+t|>0
$$
as $\log|t|>0$ for $t\in(-2,-1)$.

Since in $-2<t<-1$, $g'(t)=0$ is equivalent to
$$
\alpha_3t^3+\alpha_2t^2+\alpha_1t+\alpha_0=0
$$
where, $\alpha_0=F$, $\alpha_1=2E+F$, $\alpha_2=3D+2E$ and $\alpha_3=3D+1$.

Approximating  $\alpha_3,\alpha_2,\alpha_1,\alpha_0$ to numerical values and using Descartes' rule of signs,
it can be observed that $g'(t)$ has two positive real roots and a negative real root.

Clearly $g'(-2)g'(-1.2)<0$ implies the negative real root lies in the required interval $(-2,-1)$, which we denote as $t_0\simeq -1.50615\ldots$.

Now, $g''(t_0)>0$ justifies that $t_0$ is the minimum value.

By means of simple computation, we conclude that $g(t_0)>0$ which implies
$$
\log|G_3(t+1)|>0 \quad \mbox{ for } k=2.
$$

Again $\log|t|>0$ for $t\in(-2,-1)$. which proves that the statement is true for $k=2$.
Let the statement be true for $k=m\geq2$. Then
\begin{align*}
\log|G_3(t+1)|&\geq -\dfrac{m(m-1)(m+1)}{6}\log|t| \\
&\geq -\dfrac{m(m+1)(m+2)}{6}\log|t| \quad (\mbox{ as } m-1<m+2)
\end{align*}
which shows that the statement is true for $k=m+1$.

Hence by induction hypothesis we have
\begin{align}\label{eqn:logG3-t-k-k+1}
 \log|G_3(t+1)|+\dfrac{k(k-1)(k+1)}{6}\log|t|\geq0 \quad \mbox{ for } t\in(-k,-k+1).
\end{align}

Therefore, $\tilde{d}_3(t)\geq 0$ for $t<0$.
\end{proof}

\begin{remark}
$\tilde{d}_3(t)$ can be expressed as
\begin{align}\label{eqn:d3(t)-expression}
  \tilde{d}_3(t)=\left\{
                 \begin{array}{lll}
                  - \dfrac{\log|G_3(t+1)|+n(|t|)\log|t|}{t^3((\log|t|)^2+\pi^2)} , & \hbox{for $t<0$;} \\
                   0, & \hbox{for $t\geq 0$;}\\
                   \end{array}
               \right.
\end{align}
where
\begin{align}\label{eqn:n(t)-value}
n(t)=\dfrac{[t]([t]+1)([t]+2)}{6},\quad \mbox{ for } t>0.
\end{align}
\end{remark}

\begin{theorem}\label{Thm:lim-inf-f-3}
We have
$$\lim \inf \Im f_3(z)\geq 0 \quad \mbox{ and }\quad \Im f_3(z)\rightarrow \pi\tilde{d}_3(t)$$
as $z\rightarrow t\in\mathbb{R}$ within the upper half plane.
\end{theorem}
\begin{proof}
First we will calculate the imaginary part of $f_3(z)$ as $z\rightarrow t$ for $t\in(-k,-k+1)$,
$k\geq1$ within the upper half plane  with the help of Lemma \ref{lemma:log-Gamma3-z-t}.

Let $t\in (-k,-k+1)$ for $k\geq 1$. Then
\newline
$
\lim_{\begin{array}{c} z\rightarrow t \\\Im z>0\end{array}}f_3(z)
$
\begin{align*}
&= \dfrac{\log|G_3(t+1)|-\dfrac{i\pi}{6}k(k-1)(k+1)}{t^3(\log|t|+i\pi)}  \\
&=\dfrac{\log|G_3(t+1)|\log|t|+\dfrac{\pi^2}{6}k(k-1)(k+1)}{t^3((\log|t|)^2+\pi^2)}
    -i\pi \dfrac{\log|G_3(t+1)|+\dfrac{k(k-1)(k+1)}{6}\log|t|}{t^3((\log|t|)^2+\pi^2)}.
\end{align*}
Hence,
\begin{align*}
  \lim_{y\rightarrow 0^{+}}\Im f_3(t+iy)= -\pi \dfrac{\log|G_3(t+1)|+\dfrac{k(k-1)(k+1)}{6}\log|t|}{t^3((\log|t|)^2+\pi^2)}
\end{align*}

For $t=-k,\quad k=1,2,3,\ldots $ , we have
$$
|f_3(-k+iy)|\geq \dfrac{|\log|G_3(-k+1+iy)||}{|-k+iy|^3|\Log(-k+iy)|}\longrightarrow \infty
$$
for $y\rightarrow 0^{+}$ because $G_3(z)$ has poles at $z=0,-1,-2,\ldots$.

Since $f_3(z)$ is real on the positive real axis, $\Im f_3(z)\rightarrow0 $ as $z\in t>0$.
Clearly, with the help of  \eqref{eqn:logG3-t-k-k+1} we have
$$\lim \inf \Im f_3(z)\geq 0 \quad \mbox{ as }z\rightarrow t\in \mathbb{R}\setminus\{0\}.$$

Now we will analyze the behaviour of $f_3(z) =\dfrac{\log G_3(z+1)}{z}\cdot\dfrac{1}{z^2\Log z}$ at $z=0$.

Since,
\begin{align*}
  \lim_{z\rightarrow 0}\dfrac{\log G_3(z+1)}{z}=(\log G_3)'(1)=F<0,
\end{align*}
the behaviour of $f_3(z)$ at origin is determined by the behaviour of $F/(z^2\Log z)$.

Let $z=re^{i\theta}$. Then
\begin{align*}
  \Im\left( \dfrac{1}{z^2\Log z} \right)&=-\dfrac{\theta\cos 2\theta+\log r\sin 2\theta}{r^2((\log r)^2+\theta^2)}\\
&=-\dfrac{\pi}{r^2((\log r)^2+\pi^2)}\quad \mbox{ for } z\mbox{ within upper half plane.}
\end{align*}

Hence, $\lim \inf \Im f_3(z)\geq 0,$ as $z\rightarrow t\in\mathbb{R}$ within the upper half plane.

Now with the help of Lemma \ref{lemma:log-Gamma3-z-t} and Lemma \ref{lemma:d3-t-value},
we can conclude that $\Im f_3(z)\rightarrow \pi\tilde{d}_3(t)$ as $z\rightarrow t\in\mathbb{R}$ within the upper half plane.
\end{proof}

In the next theorem our main objective is to find integral representation of the function
$$
F_3(z)=\dfrac{\log G_3(z+1)}{z^3\Log z},\quad z\in \mathbb{C}\setminus (-\infty,0].
$$

We will use the following result of N. I. Akhiezer \cite{Akheizer65_book} to prove the next result.
\begin{lemma}[{\cite[p. 127]{Akheizer65_book}}]\label{lemma:Nevanlinna-Pick-iff}
The function $f(z)$ has the representation
$$
f(z)=\alpha+\int_0^{\infty}\dfrac{d\nu(t)}{t-z}, \quad (0<\arg z<2\pi)
$$
where $\alpha\geq 0$ and $\nu(t)$ is non-decreasing function and $$\int_0^{\infty}\dfrac{d\nu(t)}{1+t}<\infty,$$
if and only if
\begin{enumerate}[(i)]
\item $f(z)$ belongs to Nevanlinna-Pick class.
\item The function $f(z)$ is continuous and positive on the negative half of the real axis.
\end{enumerate}
\end{lemma}

\begin{remark}
Lemma $\ref{lemma:Nevanlinna-Pick-iff}$ provides the equivalence between a Pick function and
its Stieltjes integral representation when the measure is concentrated on $(0,\infty)$.
The result is due to Krein \cite{KreinRekhtman38_NevanPick}.
Conditions $(i)$ and $(ii)$ are essentially as those in the well-known lemma due to L\"{o}wner.
\end{remark}

\begin{theorem}\label{Thm:Gamma3StieltjesRepr}
The function $$F_3(z)=\dfrac{\log \Gamma_3(z+1)}{z^3\Log z}$$
is a Pick function with the Stieltjes representation
\begin{align}\label{eqn:logGamma3-Stieltjes-Rep}
\dfrac{\log \Gamma_3(z+1)}{z^3\Log z}=\dfrac{1}{6}-\int_0^{\infty}\dfrac{\tilde{d}_3(-t)}{t+z}dt
\end{align}
where $\tilde{d}_3(t)$ is as given in \eqref{eqn:d3(t)-expression}.
\end{theorem}

\begin{proof}
We proceed in similar lines to the one given in \cite{Pedersen03_JCAM_doubleGammaPick}.
Using Theorem \ref{Thm:lim-inf-f-3} it can be proved that $\Im F_3$ is bounded from below by $0$ on the
real line as it is a harmonic function . Using ordinary maximum principle we can conclude that the $\Im F_3$
is bounded from below in the upper half plane.
Consequently the function $F_3$ is bounded from below in the upper half plane by the maximum principle.
Also this function has nonnegative boundary values on the real line which means it will be positive by the
maximum modulus theorem for extended boundary \cite[p. 129]{Conway78_book}( see also \cite[p. 23]{Koosis88_LogIntegral_Book}).
Nonnegativity of $\Im F_3$ in the upper half plane implies that $F_3$ is a Pick function.

It is well known that Pick functions have integral representation of the form \eqref{eqn:Pick-DefnIntgralRep}.
Moreover $\mu$ has no support where $F_3(z)$ can be analytically extended across the real line such
that $\overline{F_3(\bar z)}=F_3(z)$. This means, $\mu$ is concentrated on the negative half line.

It is easy to find $a=0$ with the help of \eqref{eqn:PickValue-a-b} and Lemma \ref{lemma:P-Gamma3-bound}.
Now using differentiation under integration we obtain that $F_3$ is increasing on the positive real line.
By Lemma \ref{lemma:log-Gamma3-limit-x}, we find that $F(x)$ tends to $1/6$ as $x\rightarrow\infty$.
Hence, the function
$$
F(z):=\dfrac{1}{6}-F_3(z)
$$
is analytic in $\mathbb{C}\setminus (-\infty,0]$.
Using Theorem \ref{Thm:lim-inf-f-3} we can conclude that $F(z)$ has negative imaginary part in the upper half plane.
Further, with the help of Lemma \ref{lemma:log-Gamma3-limit-x}, it is easy to see that the function $F(z)$ is positive on the positive half line.

Now we apply Lemma \ref{lemma:Nevanlinna-Pick-iff} to obtain the following Stieltjes integral representation
 $$F(z)=\beta+\int_0^{\infty}\dfrac{d\nu(t)}{t+z},$$
where $\beta$ is a nonnegative real number and $\nu$ is a positive Borel measure satisfying $\displaystyle\int \dfrac{d\nu(t)}{t+1}<\infty.$

With the help of Lemma \ref{lemma:log-Gamma3-limit-x} we can conclude that $\beta=0$. It can be
further observed that $\nu(t)=\mu(-t)$, where the measure $\mu$ is defined as in \eqref{eqn:PickValue-mu}.

Considering $\tilde{d}_3(t)$ given in \eqref{eqn:d3(t)-expression} and
using Theorem \ref{Thm:lim-inf-f-3} it can be proved that for any continuous function $f$ of
compact support not containing the origin, (see also \cite[Lemma 4.1]{Pedersen07_ETNA_PickNegativeZero})
\begin{align*}
\lim_{y\rightarrow 0^{+}}\dfrac{1}{\pi}\int_{-\infty}^{\infty} f(t)\Im F_3(t+iy)dt=\int_{-\infty}^{\infty}f(t)\tilde{d}_3(t)dt.
\end{align*}
Here the limit is taken in the sense of vague topology which is a property of Pick functions \cite{Donoghue74_monograph}.

It remains to check the behaviour near origin. Since $\mu$ has density $\tilde{d}_3(t)$ on the negative
line $(-\infty,0)$ and supported on the closed half line, $\mu$ can be expressed as $\tilde{d}_3(t)dt+c\epsilon_0$,
for some $c\geq0$, where $\epsilon_0$ represents point mass at $0$. Consequently
$$
xF_3(x)=\dfrac{x}{6}-\int_0^{\infty}\dfrac{x\tilde{d}_3(-t)}{t+x}dt+c
$$
Taking $x\rightarrow0$ we have $c=0$, which gives the required result and completes the proof.
\end{proof}

\begin{corollary}
The derivative of $f_3(x)$ is completely monotone for all $x>0$.
\end{corollary}
\begin{proof}
Since $f_3$ is increasing on the positive real line. Using repeated differentiation under the sign of integration we get
\begin{align*}
(-1)^{n+1}f_3^{(n)}\geq0 \quad \mbox{ for all } x>0.
\end{align*}
Consequently, $f_3^{(2k)}(x)<0$ and $f_3^{(2k-1)}(x)>0$ for all $k\geq 1$ and all $x>0$.
which implies that derivative of $f_3(x)$ is completely monotone for all $x>0$.
\end{proof}

\section{Concluding remarks}\label{sec:Remarks-G-n-Pick}

Results given for  $f_3(z)$ in Section $\ref{Sec: triple-Gamma-Pick-results}$
and in general for $f_n(z)$ in Section $\ref{Sec:proof of G-n}$ can be discussed in other framework as well.
For example we consider the situation $\mathcal{A}=(-\infty,0]\cup\{1\}$.
The volume $V_n$ of the unit ball in $\mathbb{R}^n$ is given by \cite{BergPedersen11_PickGammaUnitball}
$$
V_n=\dfrac{\pi^{n/2}}{\Gamma(1+n/2)}, \quad n=1,2,\ldots.
$$
Hence, if we consider the function  $g(z)=\left(\dfrac{\pi^{z^3/2}}{G_3(z+1)} \right)^{1/(z^3\Log z)}$ after modifying $V_n$,
we have the following result.

\begin{theorem}\label{Thm:G-3-unit-ball-Pick}
Let
\begin{align*}
g_3(z)=\left(\dfrac{\pi^{z^3/2}}{G_3(z+1)} \right)^{1/(z^3\Log z)} \quad \mbox{ for } z\in \mathbb{C}\setminus \mathcal{A}.
\end{align*}
Then $\dfrac{1}{6}+\log g_3(z+1)$ is a Stieltjes function with integral representation
\begin{align*}
\log g_3(z+1)=-\dfrac{1}{6}+\dfrac{\log\sqrt{\pi}}{\Log(z+1)}+\int_1^{\infty}\dfrac{\tilde{d}_3(1-t)}{t+z}dt
\end{align*}
where $\tilde{d}_3(t)$ is defined as in \eqref{eqn:d3(t)-expression}.
\end{theorem}
\begin{proof}
We have
\begin{align*}
\log g_3(z+1)=\dfrac{\log\sqrt{\pi}}{\Log(z+1)}-\dfrac{\log G_3(z+2)}{(z+1)^3\Log(z+1)}
\end{align*}
Now from \eqref{eqn:logGamma3-Stieltjes-Rep} we have
\begin{align*}
  \log g_3(z+1)
= -\dfrac{1}{6}+ \dfrac{\log\sqrt{\pi}}{\Log(z+1)}+\int_1^{\infty} \dfrac{\tilde{d}_3(1-t)}{t+z}dt
\end{align*}

It is already proved that $1/\Log(z+1)$ is a Stieltjes function \cite[p. 130]{BergForst75_book} with its integral representation
$$
\dfrac{1}{\Log(z+1)}=\int_1^{\infty}\dfrac{dt}{(z+t)((\ln(t-1))^2+\pi^2)}
$$
Which implies that $\dfrac{1}{6}+\log g_3(z+1)$ is a Stieltjes function.
\end{proof}

In a similar way the following result can be proved.

\begin{theorem}
Let
\begin{align*}
g_n(z)=\left(\dfrac{\pi^{z^n/2}}{G_n(z+1)} \right)^{1/(z^n\Log z)} \quad \mbox{ for } z\in \mathbb{C} \setminus \mathcal{A}.
\end{align*}
Then $\dfrac{1}{n!}+\log g_n(z+1)$ is a Stieltjes function with integral representation
\begin{align*}
\log g_n(z+1)=-\dfrac{1}{n!}+\dfrac{\log\sqrt{\pi}}{\Log(z+1)}+\int_1^{\infty}\dfrac{\tilde{d}_n(1-t)}{t+z}dt
\end{align*}
where $\tilde{d}_n(t)$ is defined as in Conjecture \ref{conj:f-n-Pick}.
\end{theorem}

In \cite{BergPedersen11_PickGammaUnitball} the function
$$
F_a(x)=\dfrac{log \Gamma(x+1)}{x\log(ax)}, \qquad a\geq 0
$$
is considered and proved as a Pick function for $a\geq 1$.
In particular, the case $a=1$ is a Bernstein function is proved in \cite{BergPedersen01_MonotoneGamma}.
If Conjecture $\ref{conj:f-n-Pick}$ is true, it will be interesting to study the same by replacing $z$ by $az$ for
the function $f_n(z)$ given in Conjecture $\ref{conj:f-n-Pick}$.
Also extending Theorem $\ref{Thm:Gamma3StieltjesRepr}$ for the corresponding case $F_3(az)$ is expected to provide
similar results.

In \cite{BergPedersen01_MonotoneGamma} and \cite{BergPedersen11_PickGammaUnitball}, Berg and Pedersen have studied the completely
monotonic behaviour of respective functions. In similar lines,
it would be interesting to study the completely monotone behaviour of the function
$g_3(z)$ given in Theorem $\ref{Thm:G-3-unit-ball-Pick}$.

It is known that completely monotone functions are related to Pick functions.
For example,
if $(-1)^n f_n(x)\geq 0$ then $f(x)$ is completely monotone implies $f(x)$ is a Pick function (see \cite{BariczSwami14_qhyper} for details).
Further Baricz in \cite{Baricz08_PAMS_Turan_hyper} has questioned if such $f(x)$ is a Bernstein function.
Also, Bernstein functions are positive and have completely monotone derivative.
Hence we conclude with the following open problem.
\begin{problem}
For $n\in {\mathbb{N}}$, is $f_n(z)$ a Bernstein function?
\end{problem}

\end{document}